\documentclass[12pt]{amsart}
\usepackage{amsmath}
\usepackage{amssymb}
\usepackage{geometry}
\usepackage{graphicx}
\usepackage{amsthm}
\usepackage{setspace}
\usepackage[final]{showkeys}

\DeclareMathOperator{\interior}{interior}

\newtheorem{thm}{Theorem}[section]
\newtheorem{lemma}{Lemma}[section]
\newtheorem{prop}{Proposition}[section]
\newtheorem{cor}{Corollary}[section]

\theoremstyle{definition}

\newtheorem{remark}{Remark}[thm]
\newtheorem{example}{Example}[thm]

\newcommand{\R}{\mathbb{R}}

\begin{document}

\title[Borel--Cantelli lemmas for non-uniformly hyp.\ dyn.\ sys.]{A note on Borel--Cantelli lemmas for non-uniformly hyperbolic dynamical systems}
\author{Nicolai Haydn}
\address{Nicolai Haydn\\ Department of Mathematics, University of
  Southern California\\ Los Angeles, CA, 90089}
\email{nhaydn@usc.edu}
\author{Matthew Nicol}
\address{Matthew Nicol\\Department of Mathematics, University of
  Houston\\ Houston, TX, 77204-3008}
\email{nicol@math.uh.edu}
\author{Tomas Persson}
\address{Tomas Persson\\Department of Mathematics, Lunds Tekniska
  H\"{o}gskola\\Box 118, 22100 Lund}
\email{tomasp@maths.lth.se}
\author{ Sandro Vaienti}
\address{Sandro Vaienti\\ UMR-6207 Centre de Physique Th\'eorique, CNRS,
Universit\'es d'Aix-Marseille I, II, Universit\'e du Sud, Toulon-Var and FRUMAM, F\'ed\'ederation de
Recherche des Unit\'es de Math\'ematiques de Marseille\\ CPT, Luminy Case 907, F-13288
Marseille Cedex 9, France}
\email {vaienti@cpt.univ-mrs.fr}

\thanks{MN  and TP wish to thank the Institut Mittag-Leffler for support and hospitality. MN  wishes to thank 
the Institut de Math\'ematiques Luminy also for  support and hospitality.}

\begin{abstract}
  Let $(B_{i})$ be a sequence of measurable sets in a probability
  space $(X,\mathcal{B}, \mu)$ such that $\sum_{n=1}^{\infty} \mu
  (B_{i}) = \infty$.  The classical Borel--Cantelli lemma states that
  if the sets $B_{i}$ are independent, then $\mu (\{x \in X : x \in
  B_{i} \text{ infinitely often (i.o.) } ) = 1$.

  Suppose $(T,X,\mu)$ is a dynamical system and $(B_i)$ is a sequence
  of sets in $X$.  We consider whether $T^i x\in B_i$ for $\mu$
  a.e.\ $x\in X$ and if so, is there an asymptotic estimate on the
  rate of entry. If $T^i x\in B_i$ infinitely often for $\mu$
  a.e.\ $x$ we call the sequence $B_i$ a Borel--Cantelli sequence.  If
  the sets $B_i:= B(p,r_i)$ are nested balls about a point $p$ then
  the question of whether $T^i x\in B_i$ infinitely often for $\mu$
  a.e.\ $x$  is often called the shrinking target
  problem.

 We show,  under certain assumptions on the measure $\mu$, that for balls $B_i$ if 
  $\mu (B_i)\ge i^{-\gamma}$,  $0<\gamma <1$, then a sufficiently
  high polynomial rate of decay of correlations for Lipschitz
  observations implies that the sequence is Borel--Cantelli. If $\mu
  (B_i)\ge \frac{C\log i}{i}$ then exponential decay of correlations
  implies that the sequence is Borel--Cantelli.  If it is only assumed
  that $\mu (B_i) \ge \frac{1}{i}$ then we give conditions in terms of
  return time statistics which imply that for $\mu$ a.e.\ $p$
  sequences of nested balls $B(p,1/i)$ are
  Borel--Cantelli. Corollaries of our results are that for planar
  dispersing billiards and Lozi maps $\mu$ a.e.\ $p$ sequences of
  nested balls $B(p,1/i)$ are Borel--Cantelli. We also give
  applications of these  results to a variety of non-uniformly hyperbolic dynamical
  systems.
\end{abstract}

\maketitle

\section{Introduction} 

Suppose $(X,\mathcal{B},\mu)$ is a probability space.  For a
measurable set $A \subset X$, let $1_{A}$ denote the characteristic
function of $A$. The classical Borel--Cantelli lemmas (see for
example~\cite[Section~4]{Durret}) state that
\begin{itemize}
  \item[(1)] if $(A_{n})_{n=0}^{\infty}$ is a sequence of measurable
    sets in $X$ and $\sum_{n=0}^{\infty} \mu (A_{n}) < \infty$ then
    $\mu ( x \in A_{n}\ \mathrm{i.o.}) = 0$.
  \item[(2)] if $(A_{n})_{n=0}^{\infty}$ is a sequence of independent
    sets in $X$ and $\sum_{n=0}^{\infty} \mu(A_n) = \infty$, then for
    $\mu$ a.e.\ $x\in X$
    \[
    \frac{S_n (x) }{E_n} \to 1
    \]
    where $S_n (x)=\sum_{j=0}^{n-1} 1_{A_j} (x)$ and $E_n
    =\sum_{j=0}^{n-1} \mu (A_j)$.
\end{itemize}

Suppose now $T \colon X \to X$ is a measure-preserving transformation
of a probability space $(X, \mu)$.  If $(A_n)$ is a sequence of sets
such that $\sum_{n} \mu (A_{n})=\infty$ it is natural in many
applications to ask whether $T^{n} (x) \in A_{n}$ for infinitely many
values of $n$ for $\mu$ a.e.\ $x\in X$ and, if so, is there a
quantitative estimate of the asymptotic number of entry times? For
example, the sequence $(A_{n})$ may be a nested sequence of balls
about a point, a setting which is often called the shrinking target
problem.

The property $\lim_{n\to \infty} \frac{S_n (x)}{E_n}=1$ for $\mu$
a.e.\ $x\in X$ is often called the Strong Borel--Cantelli (SBC) in
contrast to the Borel--Cantelli (BC) property that $S_n (x)$ is
unbounded for $\mu$ a.e.\ $x\in X$.

W. Phillipp~\cite{Phillipp} established the SBC property for sequences of intervals
in the setting of certain maps of the unit interval including the $\beta$
transformation, the Gauss transformation and smooth uniformly
expanding maps.

There have been some results on Borel--Cantelli lemmas for uniformly
hyperbolic systems in higher dimensions.  Chernov and
Kleinbock~\cite{Chernov_Kleinbock} establish the SBC property for
certain families of cylinders in the setting of topological Markov
chains and for certain classes of dynamically-defined rectangles in
the setting of Anosov diffeomorphisms preserving Gibbs measures.
Dolgopyat~\cite{Dolgopyat} has related results for sequences of balls
in uniformly partially hyperbolic systems preserving a measure
equivalent to Lebesgue which have exponential decay of correlations
with respect to H\"older observations.

Kim~\cite{Kim} has established the SBC property for
sequences of intervals in the setting of $1$-dimensional
piecewise-expanding maps $T$ with $\frac{1}{|T^{'}|}$ of bounded
variation.

Kim uses this result to prove some SBC results for non-uniformly
expanding maps with an indifferent fixed point.  In particular, he
considers intermittent maps of the form
\begin{equation}\label{e:intermittent}
T_{\alpha} (x) =
\begin{cases}
x(1 + 2^{\alpha} x^{\alpha}) &\text{if } 0 \le x < \frac{1}{2};\\
2x-1 &\text{if } \frac{1}{2} \le x \le 1.
\end{cases}
\end{equation}
These maps are sometimes called Liverani--Saussol--Vaienti
maps~\cite{LSV}. If $0<\alpha<1$ then $T_{\alpha}$ admits an invariant
probability measure $\mu$ that is absolutely continuous with respect
to Lebesgue measure $m$.  We  use the notation $x_n \sim y_n$ to denote $\lim_{n\to \infty}\frac{x_n}{y_n}=1$.
 The measure $\mu$ has an unbounded density $h(x)\sim C  x^{-\alpha}$ near $0$.  Kim shows that if $(I_n)$ is a sequence of
intervals in $(d,1]$ for some $d>0$ and $\sum_n \mu (I_n)=\infty$ then
  $I_{n}$ is an SBC sequence if (a) $I_{n+1}\subset I_n$ for all $n$
  (nested intervals) or (b) $\alpha < (3 - \sqrt{2}) / 2$. Kim shows
  that the condition $I_n \subset (d,1]$ for some $d>0$ is in some
    sense optimal (with respect to the invariant measure $\mu$) by
    showing that setting $A_n=[0,n^{-1 / (1-\alpha)})$ gives a
      sequence such that $\sum_n \mu (A_n)=\infty$ yet $T_{\alpha}^{n}
      (x)\in A_{n}$ for only finitely many values of $n$ for $\mu$
      a.e.\ $x\in [0,1]$.

For the same class of maps $T_{\alpha}$, Gou\"{e}zel~\cite{Gouezel}
considers Lebesgue measure $m$ (rather than the invariant probability
measure $\mu$) and shows that if $(I_n)$ is a sequence of intervals
such that $\sum_n m (I_{n})=\infty$ then $(I_n)$ is a BC
sequence.  Assumptions (a) or (b) of Kim are not necessary for
Gou\"{e}zel's result.  Gou\"{e}zel uses renewal theory and obtains BC
results but not SBC results.

At the end of this section we give  an example of  an  intermittent type map which preserve Lebesgue measure $m$
yet  for which there exists a  sequence of nested intervals $I_n$, $\sum_n m(I_n)=\infty$ yet $(I_n)$ is not BC. Such 
maps may have arbitrarily high polynomial rate of decay of correlations for H\"older observations.

In the context of flows, Maucourant~\cite{Maucourant} has proved the BC property for nested
balls in the setting of geodesic flows on hyperbolic manifolds of
finite volume.

Recently, Gupta et al.~\cite{GNO} proved the SBC property for sequences of
intervals in the setting of Gibbs--Markov maps and also sequences of
nested balls in one-dimensional maps modeled by Young Towers. They
also gave some applications to extreme value theory of deterministic
systems, in particular the almost sure behavior of successive maxima
of observations on such systems.

Fayad~\cite{Fayad} has given an example of an analytic area preserving
map of the three-dimensional torus which is mixing of all orders, yet for which 
there exists a sequence of nested balls $(A_n) $ such that $E_n$
diverges yet the sequence $(A_n)$  is not BC.

In this short note we establish dynamical Borel--Cantelli lemmas using
elementary arguments, we try to avoid  dynamical assumptions
beyond decay of correlations as much as we can. Our results are phrased in terms of the 
interplay between the measure $\mu (A_n)$ of the sets and the rate of decay of correlations of
observations in various norms. For sets of measure $\mu (A_n) \le \frac{1}{n}$ we mainly restrict to the 
shrinking target problem and  need to make some 
dynamical assumptions  in the guise of return time statistics. We define
$E(\phi):=\int_{X} \phi \, d \mu$ for the expectation of an integrable observation on a
dynamical system $(T,X,\mu)$ where $X$ is a metric and probability
space.

In applications it is common to have decay estimates  for observations on a  dynamical system in various Banach space norms, for example:
\begin{itemize}
  \item[(a)] Bounded variation (BV) versus $L^1$,
    \[
    |E(\phi~\psi\circ T^m)-E(\phi)E(\psi)|\le 
    p(m)\|\phi\|_{BV}~\|\psi\|_{1},
    \]

  \item[(b)] Lipschitz versus $L^{\infty}$,
    \[
    |E(\phi~\psi\circ T^m)-E(\phi)E(\psi)|\le 
    p(m)\|\phi\|_{\mathrm{Lip}}~\|\psi\|_{\infty},
    \]

  \item[(c)] Lipschitz versus Lipschitz,
    \[
    |E(\phi~\psi\circ T^m)-E(\phi)E(\psi)|\le 
    p(m)\|\phi\|_{\mathrm{Lip}}~\|\psi\|_{\mathrm{Lip}},
    \]
\end{itemize}
where $p(m)$ is a rate function which tends to zero in $m$. H\"{o}lder
norms are sometimes considered rather than Lipschitz but it is no
essential loss to only consider Lipschitz.

In fact $BV$ versus $L^1$ is not so common for non-uniformly
hyperbolic systems i.e. those with a stable foliation.  It is implicit in Kim~\cite{Kim} and explicitly
stated in Gupta et al.~\cite{GNO} that summable rate of decay
(i.e.\ $\sum_m p(m)<\infty$) for the norms $BV$ versus $L^1$ implies the
SBC property for any sequence of balls $B_i$ such that $E_n$ diverges.

\begin{prop}[\cite{Kim,GNO}]
  Suppose $(T,X,\mu)$ has summable decay of correlations with respect
  to $BV$ versus $L^1$ in the sense that for $\phi \in BV$, $\psi \in
  L^1 (\mu)$ 
  \[
  |E(\phi~\psi\circ T^m )-E(\phi)~E(\psi) |\le p(m)\|\phi\|_{BV} \|\psi\|_{1}
  \]
  where $\sum_m p(m)<\infty$. 
 If $(B_{i})$ is a sequence of balls in $X$ and $\sum_{i=0}^{\infty}
 \mu (B_{i}) = \infty$ then $(B_i)$ is SBC.
\end{prop}
The proof of this result is a straightforward application of a
condition for SBC by Sprindzuk~\cite{Sprindzuk} that we will soon
state.

In this paper we will mainly consider decay of correlation of
Lipschitz versus Lipschitz. The
modifications of our results for H\"{o}lder versus H\"{o}lder, Lipschitz versus $L^{\infty}$  or
H\"{o}lder versus $L^{\infty}$ are straightforward.

We will  often use a proposition of Sprindzuk~\cite{Sprindzuk}:

\begin{prop}\label{prop:sprindzuk}
  Let $(\Omega,\mathcal{B},\mu)$ be a probability space and let $f_k
  (\omega) $, $(k=1,2,\ldots )$ be a sequence of non-negative $\mu$
  measurable functions and $g_k$, $h_k$ be sequences of real numbers
  such that $0\le g_k \le h_k \le 1$, $(k=1,2, \ldots,)$.  Suppose there exist $C>0$ such that 
  \begin{equation} \label{eq:sprindzuk}
    \tag{$*$} \int \left(\sum_{m<k\le n}( f_k (\omega) - g_k)
    \right)^2\,d\mu \le C \sum_{m<k \le n} h_k
  \end{equation}
  for arbitrary integers $m <n$. Then for any $\epsilon>0$
  \[
  \sum_{1\le k \le n} f_k (\omega) =\sum_{1\le k\le n} g_k (\omega)  +
  O (\theta^{1/2} (n) \log^{3/2+\epsilon} \theta (n)
  )
  \]
  for $\mu$ a.e.\ $\omega \in \Omega$, where $\theta (n)=\sum_{1\le k
    \le n} h_k$.
\end{prop}

\begin{example}
Cristadoro et.~al.~\cite{CHMV} consider an intermittent map $T$ of the interval $[-1,1]$ (actually on the unit circle
as $1$ and $-1$ are identified) with an unbounded derivative at the origin. The map is implicitly defined by the equation

\[ x = \left\{ \begin{array}{ll}
         \frac{1}{2\gamma} (1+T(x))^{\gamma})^{\gamma} & \mbox{if $0\le x \le \frac{1}{2\gamma} $};\\
        T(x) +\frac{1}{2\gamma} (1-T(x))^{\gamma} & \mbox{if $\frac{1}{2\gamma} \le x\le1$}.\end{array} \right. 
 \] 
 and extended as an odd function so that  $T(-x)=-T(x)$. See Figure 1.

 We assume $\gamma >1$. Let $\tau=\frac{1}{\gamma-1}$ and use the notation $x_n \sim y_n$ to
 denote $\lim_{n\to \infty}\frac{x_n}{y_n}=1$.

 The map preserves Lebesgue measure 
 $m$ and is mixing with  a polynomial rate of decay of correlations for H\"older observations. In fact
 \[
 | \int \phi\circ T^n\psi dm-\int \phi dm \int \psi dm | \le C(\phi,\psi) n^{-\tau}
 \]
 for all H\"older $\phi$, $\psi \in L^{\infty}(m)$.
 
 We  define $T_{+}:=T_{|(0,1)}$, $T_{-}:=T_{I(-1,0)}$, $a_{0-}=\frac{-1}{2\gamma}$, $a_{-i}=T_{-}^{-i}a_{0-}$ and $b_i=T_{+}^{-1}a_{-(i-1)}$.
 Note that $T^{-1} (-1,a_{-n})=(0,b_{n+1})\cup (-1,a_{-n-1})$.  
 
 It is shown in ~\cite[Lemma 2]{CHMV} that $m(-1,a_{-n})\sim ( 2\gamma\tau)^{\tau}n^{-\tau}$, $m(0,b_n)\sim\frac{1}{2\gamma}(2\gamma\tau)^{\gamma\tau}(n-1)^{-\gamma\tau}$.

If we choose $\gamma>2$ then $0<\tau <1$ so $\sum_{n>0}m(-1,a_{-n})$ diverges. Since $m(0,b_n)\sim \frac{1}{2\gamma}(2\gamma\tau)^{\gamma\tau}
(n-1)^{-\gamma\tau}$ and $\gamma\tau>1$, $\sum_{n>0} m(0,b_n)$ converges.  The only way for the orbit of a  generic point under the map $T$
to enter $(-1,a_{-n})$ infinitely often is to enter $(0,b_n)$ infinitely often. Since $\sum_{n>0} m(0,b_n)$ converges, $m$ a.e. satisfies
$T^nx \in (-1,a_{-n})$ for only finitely many $n$, although $\sum_{n>0}m(-1,a_{-n})$ diverges.This is an example of a sequence of 
nested sets $A_n:= (-1,a_{-n})$ such that $\sum_{n>0}m(A_n)$ diverges yet $T^n x\in A_n$ at most finitely many times for $m$ a.e. $x$. In contradiction
to the example of Kim, the map $T$ preserves Lebesgue measure rather than a measure with an unbounded density. 
Note also  that taking $\gamma-1:=\delta>0$ the rate of decay of correlation is $Cn^{-\frac{1}{\delta}}$ is at an  arbitrarily high polynomial rate.

\end{example}

\setlength{\unitlength}{0.5\textwidth}

\begin{figure}
  \begin{minipage}{0.5\textwidth}
    \begin{center}
      \includegraphics[width=\unitlength]{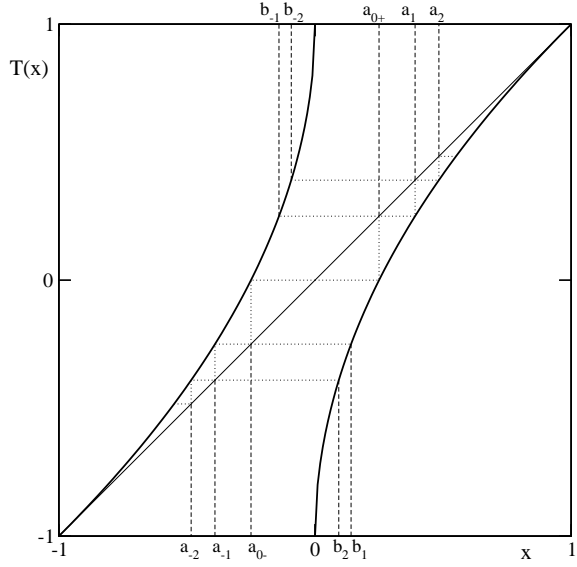}
      \caption{The graph of $T$.}
      \label{fig:map}

    \end{center}
  \end{minipage}
 \end{figure}

\section{Assumptions}

Suppose $(T,X,\mu)$ is an ergodic measure preserving map of a
probability space $X$ which is also a metric space.  We assume:
\begin{itemize}
  \item[(A)] For all Lipschitz functions $\phi,\psi$ on $X$ we have
    summable decay of correlations in the sense that there exists
    $C>0$, and a rate $p(k)\rightarrow \infty$, $\sum_k p(k)<\infty$
    (both independent of $\phi,\psi$) such that
    \[
    |E( \phi~\psi\circ T^k) -E(\phi) E(\psi) | <  p(k) \| \phi \|_{\mathrm{Lip}} \| \psi \|_{\mathrm{Lip}}.
    \]

    \item[(B)] There exists $r_0>0$, $0<\delta <1$ such that for all
      $p\in X$ and $\epsilon< r \le r_0 $
      \[
      \mu\{ x: r< d(x,p) < r+\epsilon \} < \epsilon^{\delta}. 
      \]
\end{itemize}
\begin{remark}\label{remark:delta}
If the balls $B_i=B(p,r_i)$, ($r_i \to 0$), are nested balls centered at a point $p$  then we would only require there exist $\delta (p)>0$, $r_0(p)>0$ such that 
$ \mu\{ x: r< d(x,p) < r+\epsilon \} < \epsilon^{\delta (p)}$ for all $\epsilon< r \le r_0 $.

\end{remark}

\begin{remark}\label{rectangles}
We will call the sets $B_i$ `balls', but our results extend to any shapes such that the indicator function of the 
set may be approximated closely in the $L^1$ norm by a Lipschitz function of reasonable Lipschitz norm, for example 
our results extend immediately to rectangles of bounded side ratio.

\end{remark} 

If $p(k)\le C \alpha^k$ for some  constants $C>0$, $0<\alpha<1$ then we say $T$ has
exponential decay of correlations. We will also consider polynomial
decay, where $p(k)\le C k^{-m}$ for some constants $C>0$, $m >0$.

Our results will be formulated in terms of the measure of  the balls $B_i$ and the 
rate of decay of correlation. For balls $B_i$ such that $\mu (B_i)=\frac{1}{i}$ we make additional assumptions 
on return time distributions and short return times, detailed in later sections. 

\section{Sequences of sets  $(B_i)$ such that  $\mu (B_i)\ge i^{-\gamma}$ for $0<\gamma<1$.}

Define
\[
S_n (x)=\sum_{i=0}^{n-1} {1_{B_k}}\circ T^k (x)
\]
and 
\[
E_n=\sum_{1\le k\le n} \mu (B_k).
\]

\begin{thm}\label{thm:large}
Let $   ~0<\gamma<1$ and $\mu (B_i)\ge i^{-\gamma}$. Suppose $(T,X,\mu)$
satisfies (A) and (B), with $p(k) \leq C k^{-q}$. Then if $q>
\frac{2/\delta+\gamma+1}{1-\gamma}$
\[
\lim_{n\rightarrow \infty} \frac{S_n (x)}{E_n}=1
\]
for $\mu$ a.e.\ $x \in X$.
\end{thm}

\begin{remark}

The proof of Theorem~\ref{thm:large} gives the following asymptotic bounds. 
\[
 \sum_{i=0}^{n-1} {1_{B_k}}\circ T^k (x)\approx \sum_{1\le k\le n} \mu
 (B_k) + O (\theta^{1/2+\epsilon} (n) \log^{3/2+\epsilon} \theta (n) )
\]
for $\mu$ a.e.\ $x \in X$, where $\theta (n)=\sum_{1\le k \le n} \mu (B_k)$.

\end{remark}

\begin{remark}
  We note that if $\mu (B_i )=\frac{1}{i}$ then no degree of polynomial decay in the Lipschitz norm
  ensures  that the sequence $(B_i )$ is BC, even
  if the $B_i$ are nested balls.  To see this consider the
  intermittent map example of Kim~\cite{Kim}, $T_{\alpha} \colon
  [0,1]\rightarrow [0,1]$ with the balls $B_i :=[0,i^{-1 /
      (1-\alpha)})$ so that $\mu (B_i)\equiv \frac{1}{i}$. The density
    $h(x)=\frac{d\mu}{dm}$ behaves like $x^{-\alpha}$ near $0$ and the
    rate of decay of correlations with respect to Lipschitz versus
    $L^{\infty}$ functions is $n^{1-\frac{1}{\alpha}}$. The constant
    $\delta$ in the statement of the theorem (from Property~$(B)$) may
    be taken as $1-\alpha$ by observing $\int_0^{\epsilon} x^{-\alpha}
    \, dx\le C\epsilon^{1-\alpha}$. As $\alpha\to 0$, $\delta
    (\alpha)\rightarrow 1$ and $q(\alpha)\rightarrow \infty$ yet for
    each $0<\alpha <1$, the sequence $B_i=[0,i^{-1 / (1-\alpha)})$ is
      not BC.
\end{remark}

\begin{proof}
For each $k$ let $\tilde{f}_k$ be a Lipschitz function such that
$\tilde{f}_k (x)=1_{B_k} (x)$ if $x\in B_k$, $\tilde{f}_k (x)=0$ if
$d(B_k, x)>(k(\log k)^2)^{-1/\delta}$, $0\le \tilde{f}_k \le 1$ and
$\|\tilde{f}_k\|_{\mathrm{Lip}} \le (k(\log k)^{2})^{1/\delta}$.  Clearly we
may construct such functions by linear interpolation of $1_{B_k}$ and
$0$ on a region $r \le d(p_k, x)\le r+ (k(\log k)^2)^{-1/\delta}$.

In the proposition above we will take $f_k =\tilde{f}_k\circ T^k (x)$,
$g_k=E(\tilde{f}_k)$, where $E(\phi)=\int \phi \, d\mu$ for any
integrable function $\phi \in L^1 (\mu)$. The constants $h_k$ will be
chosen later.

Note that for $\mu$ a.e.\ $x\in \Omega$, $f_i (x) = 1_{B_i} (T^i x)$
except for finitely many $i$ by the Borel--Cantelli lemma as $\mu (x:
f_i (x) \not = 1_{B_i} (T^i x))= \mu (x: r_i< d (T^i x, p_i) < r_i+
(i(\log i)^2)^{-1/\delta}) < (i(\log i)^2)^{-1}$ by assumption (B).
Furthermore $\sum_k \mu (B_k)=\sum_k g_k +O(1)$.

A rearrangement of terms in \eqref{eq:sprindzuk} means that it
suffices to show that there exists a $C>0$ such that
\[
\sum_{i=m}^n \sum_{j=i+1}^n E(f_i f_j)-E(f_i)E(f_j) \le C \sum_{i=m}^n h_i
\]
for arbitrary integers $n>m$ (where $h_i$ will be chosen later).

We split each sum $\sum_{j=i+1}^n E(f_i f_j)-E(f_i)E(f_j)$ into the terms 
$$
 I_i=\sum_{j=i+1}^{i+\Delta} E(f_i f_j)-E(f_i)E(f_j)
$$
and 
$$
II_i=\sum_{j=i+\Delta+1}^{n} E(f_i f_j)-E(f_i)E(f_j)
$$
where here we put $\Delta=[i^\sigma]$ with the 
 value of $0<\sigma<1$ is to be chosen later. 

The first term $I_i$ is roughly estimated by
$$
 \sum_{j=i+1}^{i+i^{\sigma}} E(f_i
f_j)-E(f_i)E(f_j)\le \sum_{j=i+1}^{i+i^{\sigma}} E( f_i f_j)\le
\sum_{j=i+1}^{i+i^{\sigma}} \frac{2}{j^{\gamma}}\le i^\sigma\frac2{i^\gamma}.
$$

 
The second term $II_i$ we bound using the decay of correlations
\begin{eqnarray*}
II_i&=&  \sum_{j=i+i^{\sigma}}^{\infty} E(f_i f_j) - E(f_i)E(f_j) \\
&=&\sum_{j=i+i^{\sigma}}^{\infty} \int_X f_i (T^i x) f_j (T^j x) \, d\mu
  - E(f_i) E (f_j)\\
  &=& \sum_{j=i+i^{\sigma}}^{\infty} \int_X f_i ( x)
  f_j (T^{j-i} x) \, d\mu - E(f_i) E (f_j)\\
  &\le& \sum_{j=i+i^{\sigma}}^{\infty}  \|f_i\|_{\mathrm{Lip}} \| f_j \|_{\mathrm{Lip}} p(j-i) \\
  &\le& \sum_{\beta=1}^{\infty}C [(i+i^{\sigma} +\beta )(\log(i+i^{\sigma}
    +\beta ))^2]^{1/\delta} [i(\log i)^2]^{1/\delta}
  (i^{\sigma}+\beta)^{-q}.
\end{eqnarray*}
If $\sigma <1-\gamma$ then $(I_i)$ is bounded by $C i^{1-2\gamma-\rho}$
for some $\rho>0$. If $q\sigma>\frac{2}{\delta}+\gamma +1$ then $(II_i)$
is bounded by ${\mathrm{const.}} i^{-\gamma}$. Solving for $\sigma<1-\gamma$ and then for $q$ we see  that if
$q> \frac{2/\delta+\gamma+1}{1-\gamma}$ then $\sum_{j=i+1} \int E(f_i
f_j) -E(f_i) E(f_j) \le C \max \{ i^{ 1-2\gamma-\rho},i^{-\gamma} \}$.


In Sprindzuk's theorem we now take $h_i=C \max \{ i^{ 1-2\gamma-\rho},i^{-\gamma} \}$. With this
choice of $h_i$ we have $\theta (n)^{1/2} \le C \max \{
n^{(1-\gamma)/2},n^{1-\gamma-\rho}\}$.  This gives the Strong Borel--Cantelli
property  as the error term $O (\theta^{1/2} (n)
\log^{3/2+\epsilon} \theta (n) )$ is negligible with respect to
$\sum_{1\le k\le n} g_k=\sum_{1\le k\le n}E(f_k) \simeq C n^{1-\gamma}$.
\end{proof}



\section{Return time distributions and sequences $(B_i)$ such that  $\mu (B_i)\ge \frac{\log i}{i} $}

We now consider  the case $\mu (B_i) \ge \frac{\log i}{i} $.  In this case we assume exponential decay of correlations. 

\begin{thm}\label{thm:log}
Suppose $\mu (B_i)\ge \frac{\log i}{i}$ and $(T,X,\mu)$ satisfies (A) and (B). 
Then if $(T,X,\mu)$ has exponential decay of correlations  
\[
\lim_{n\rightarrow \infty} \frac{S_n (x)}{E_n}=1
\]
for $\mu$ a.e.\ $x \in X$.
\end{thm}

\begin{proof}
Let $p(k)\le C \alpha^k$ for some $0<\alpha <1$.
In  this setting we again split  $\sum_{j=i+1}^n E(f_i f_j)-E(f_i)E(f_j)$  as above
in two sums $I_i$ and $II_i$ where this time $\Delta=[(\log i)^\sigma]$, where $\sigma >1$
will be chosen below. We obtain
$$
I_i=\sum_{j=i+1}^{i+(\log i)^{\sigma}} \frac{\log j }{j}\le(\log i)^\sigma\frac{\log i}{i}=\frac{(\log i)^{\sigma+1}}i
$$
and, using the decay of correlations,
\begin{eqnarray*}
 II_i&=& \sum_{j=i+(\log i)^{\sigma}}^{N} \int_X f_i (T^i x) f_j (T^j x) \,
  d\mu - E(f_i) E (f_j) \\ &\le& C \sum_{\beta=1}^{\infty} \alpha^{(\log
    i)^{\sigma}+\beta} ((\log i)^{\sigma}+i
  +\beta)^{2/\delta}i^{2/\delta}.
  \end{eqnarray*}
As $\sigma>1$ for large enough $i$ the series above is bounded by $\frac{\log
  i}{i}$. We take $1<\sigma<2$ and  $h_k=(\log k)^{\sigma+1}/k$. Since  $\sigma <2$
the error $O (\theta (n)^{1/2} \log^{3/2 +\epsilon} \theta (n))$ is negligible
with respect to $\sum_{k=1}^n g_k\sim \sum_{k=1}^n\frac{\log k}k\sim\frac12(\log n)^2$
as $\theta(n)=\sum_{k=1}^nh_k\sim\sum_{k=1}^n\frac{(\log k)^{\sigma+1}}k\sim(\log n)^{\sigma+2}$.
\end{proof}

\subsection{Applications of Theorem~\ref{thm:large} and Theorem~\ref{thm:log}} 

We now list some dynamical systems which satisfy assumptions $(A)$ and $(B)$. Recall that 
although we have called the sets $B_i$ balls,  any sets $B_i$ which are geometrically regular in the sense
that their indicator function may be approximated in the $L^1$ norm by a Lipschitz function of reasonable Lipschitz norm
(for example rectangles)
also satisfy the conclusions of Theorem~\ref{thm:large} and Theorem~\ref{thm:log}.

\noindent {\it Dispersing billiard systems:}  satisfy assumption (B) (as the invariant measure
is equivalent to Lebesgue) and hence our results apply to these
systems if they have sufficiently high rates of decay of
correlations. The class of dispersing billiards considered by
Young~\cite{Young} and Chernov~\cite{Chernov} have exponential decay
of correlations.

\noindent {\it Lozi maps:} it is shown in Gupta et al.~\cite{GHN} that
assumption (B) is satisfied by a broad class of Lozi mappings. Lozi
mappings have exponential decay of correlations.

\noindent {\it Compact group extensions of Anosov systems:}
Dolgopyat~\cite{Dolgopyat} has shown that compact group extensions of
Anosov diffeomorphisms are typically rapid mixing i.e.\ have
superpolynomial decay of correlations. These systems satisfy also
Assumption B as they are volume preserving.

\noindent {\it Interval maps:} one-dimensional non-uniformly expanding
maps with an absolutely continuous invariant probabiility measure
$\mu$ and density $h = \frac{d\mu}{dm}\in L^{1+\delta} (m)$ for some
$\delta>0$ satisfy condition (A) and (B). For example the class of
maps considered by Collet~\cite{Collet}.

\section{Sequences of balls $(B_i)$ such that $\mu (B_i) \le i^{-1}$ and  $\sum_i \mu (B_i)=\infty$. }

If the measure of the balls $B_i$ satisfies $\mu (B_i)\le i^{-1}$ then
our arguments based solely on decay of correlations break down and to
obtain results we make more assumptions- in particular we focus on the
case of nested balls $B_{i+1}\subset B_i$ and make assumptions on
return time statistics.

In this section we focus on the shrinking target problem in which $B_i
(p) =B(p,r_i)$ the ball of radius $r_i$ about $p$ about a point $p\in
X$ such that $\mu (B(p,r_i))=i^{-1}$. Our proof and results generalize
with no change to the setting where there exists a constant $C>0$ such
that $r_{i+1}<Cr_i$ and $\mu (B(p,r_i))\ge \frac{C}{i}$.

We first observe, as remarked in Fayad~\cite{Fayad}, that the set of
points $G$ such that $T^i x \in B_i (p) $ for infinitely many $i$ has
measure zero or one. This follows as $T^{i+1} x\in B_{i+1} (p)$
implies that $T^{i} (Tx) \in B_{i} (p)$ and hence $G$ is $T$
invariant. If $(T,X,\mu)$ is ergodic this implies that $\mu (G)=0$ or
$\mu (G)=1$.

\subsection{Return time distributions.} 

We now show how a return time distribution implies Borel--Cantelli
lemmas in certain settings.  Fix $p\in X$ and set $B_i=B(p,r_i)$ where
$r_i\to 0$ so that $B_i$ is a sequence of nested balls with center
$p$. Let $\tau_{B_i}$ be a random variable defined on the set
$B_i=B(p,r_i)$ which gives the first return time i.e.\ $\tau_{B_i} (x)
=\inf \{ n\ge 1: T^n (x) \in B_i\}$.  We define the conditional
probability measure $\mu_{B_r}$ on $B_r$ by $\mu_{B_r} (A);=\frac{\mu
  (A\cap B_r)}{\mu (B_r)}$.

For many dynamical systems distributional limit laws for return time
statistics have been proven, of the form: for $\mu$ a.e.\ $p$, if $B_r
=B(p,r)$ for $r>0$ then
\[
\lim_{r\to 0} \mu_{B_r} \{ y\in B_r : \tau_{B_r} \mu (B_r) < t\} =F(t) 
\]
where $F(t)$ is a distribution function.  Commonly $F(t)=1-e^{-t}$, an
exponential law.

Let $S_n =\sum_{j=0}^{n-1} 1_{B_j}$, $E_n=\sum_{j=0}^{n-1} \mu (B_j)$
with $\lim_{n\to \infty} E_n=\infty$.  We now state a simple lemma.

\begin{lemma}\label{lemma:seq}
  Suppose that $B_i$ is a sequence of balls (not necessarily nested).
  If $E (S_n-E_n)^2 \le g(n)( E_n)^2$ for a sequence $g(n)$ such that
  $\lim_{n\to \infty} g(n)=0$ then $\limsup \frac{S_n (x)}{E_n}\ge 1$
  for $\mu$ a.e.\ $x\in X$.
\end{lemma}

\begin{proof}
  The assumptions imply that $E [(\frac{S_n}{E_n}-1)^2]\le g(n)$.
  Hence $E[ |\frac{S_n}{E_n}-1|>\epsilon]\le \frac{E
    (S_n-E_n)^2}{\epsilon^2} \le \frac{g(n)}{\epsilon^2}$ by
  Chebyshev's inequality.  If we take a subsequence $n_k$ such that
  $\sum_k g (n_k) <\infty$, then by Borel--Cantelli for $\mu$
  a.e.\ $x\in X$, $|\frac{S_n(x)}{E_n}-1|>\epsilon$ for only finitely
  many $n_k$. Let $G_{\epsilon}=\{x:
  |\frac{S_n(x)}{E_n}-1|>\epsilon~i.~o.\}$, then $\mu
  (G^c_{\epsilon})=1$.  Taking a countable sequence
  $\epsilon_m=\frac{1}{m}$ and noting $\mu (\cap_m G^c_{\epsilon_m})=1$ implies the
  result.
\end{proof}

In fact if the $B_i$ are nested and $(T,X,\mu)$ is ergodic then the
assumptions of the lemma above may be weakened to
\begin{lemma}\label{lemma:nest}
  Suppose that $B_i$ is a nested sequence of balls and $(T,X,\mu)$ is
  ergodic.  If $E (S_n-E_n)^2 < \eta ( E_n)^2$ for some $\eta <1$ then
  $T^n (x) \in B_n$ infinitely often  for $\mu$ a.e.\ $x\in X$.
\end{lemma}

\begin{proof}
  Let $G$ be the set of points $x$ such that $T^i x\in B_i$ for
  infinitely many $i$. Then $\mu (G)=1$ or $\mu (G)=0$. We assume $\mu
  (G)=0$ and derive a contradiction.
  \[
  \eta>\int_X (\frac{S_n (x)}{E_n}-1)^2 \, d\mu \geq \int_{G^c} (\frac{S_n
    (x)}{E_n}-1)^2 \, d\mu
  \]
  Let $A_N:=\{x \in G^c : T^i (x)\not \in B_i \mbox{ for all } i\ge N\}$. Then 
  \[
  \eta>\int_{G^c} (\frac{S_n (x)}{E_n}-1)^2 \, d\mu \ge \int_{A_N}
  (\frac{S_n (x)}{E_n}-1)^2 \, d\mu
  \]
  But $\lim_{n\to \infty} \int_{A_N} (\frac{S_n (x)}{E_n}-1)^2 \,
  d\mu=\mu (A_N)$. Hence $\eta> \mu (A_N)$ for all $N$, a
  contradiction.
\end{proof}

We now state our main result of this section. Let $B_i (p)$ be a
decreasing sequence of balls about a point $p$, $S_n =\sum_{j=0}^{n-1}
1_{B_j}$ and $E_n=\sum_{j=0}^{n-1} \mu (B_j)$.

\begin{thm}\label{thm:exp}
  Suppose that $(T,X,\mu)$ has exponential decay of correlations, that
  property (B) holds, that $ B_i=B(p,r_i)$ for some point $p$ and $\mu
  (B_i) \ge i^{-1}$. Also assume that
  \[
  \lim_{i\to \infty} \mu_{B_i} \{ y\in B_i : \tau_{B_i} \mu
  (B_i) < t\} =F(t)
  \]
  for some distribution function $F(t)$ such that $\lim_{t\to 0^+} F(t)=0$.
  Then for $\mu$ a.e.\ $x\in X$ 
  \[
  \limsup \frac{S_n (x)}{E_n}\ge 1.
  \]
\end{thm}

\begin{proof}

  We choose $c\ge 0$  so that  for all $i$
  \[
  (*) \sum_{\beta=1}^{\infty}C (i+c\log i +\beta )^{2/\delta}
  i^{2/\delta} \alpha^{c\log i +\beta}\le i^{-2}.
  \]
  
Since $B_j\subset B_i$ for $j>i$ 
  \[
  \sum_{j=i}^{i+c\log i} \mu (B_i \cap T^{-(j-i)} B_j) \le
  \sum_{j=i}^{i+c\log i} \mu (B_i \cap T^{-(j-i)} B_i).
  \]

 Since $\lim_{t\to 0} F(t)=0$, given $\eta>0$ there exists $t^*$ such  that $F(t)c<\eta$ for all $0<t\le t^*$.    As 
  \[
  \lim_{i\to \infty} \mu_{B_i} \{ y\in B_i : \tau_{B_i} \mu (B_i) < t\} =F(t) 
  \]
  there exists a number $n^*$ such that  for $i\ge n^*$ 
  \[
  \mu_{B_i} \{ y\in B_i : \tau_{B_i}  < -c \log \mu (B_i)\}\le F(t^*)
  \]
  Hence for $i>n^*$ 
  \[
  \sum_{j=i}^{i+c\log i} \mu (B_i \cap T^{-(j-i)} B_j) \le
  \sum_{j=i}^{i+c\log i} \mu (B_i \cap T^{-(j-i)} B_i)\le F(t^*)
  \frac{c\log i} {i}.
  \]
Recalling that  
\[
E[(S_n-E_n)^2]=2\sum_{i=1}^{n}\sum_{j>i}(\mu (B_i \cap T^{-(j-i)} B_j) -\mu(B_i) \mu (B_j)) 
+ \sum_{i=1}^n(\mu (B_i ) -\mu (B_i)^2)
\]
we note first that $\sum_{i=1}^n (\mu (B_i ) -\mu (B_i)^2) \le E_n$ and write 
\[
\sum_{i=1}^{n}\sum_{j>i}(\mu (B_i \cap T^{-(j-i)} B_j) -\mu(B_i) \mu (B_j))
=\sum_{n\ge i>n^*}\sum_{j>i}(\mu (B_i \cap T^{-(j-i)} B_j) -\mu(B_i) \mu (B_j))
\]
\[
+\sum_{1\le i<n^*}\sum_{j>i}(\mu (B_i \cap T^{-(j-i)} B_j) -\mu(B_i) \mu (B_j))
\]
 The  term 
$\sum_{1\le i<n^*}\sum_{j>i}(\mu (B_i \cap T^{-(j-i)} B_j) -\mu(B_i) \mu (B_j))$
we bound by $n^* E_n$ while 
\[
\sum_{n\ge i>n^*}\sum_{j>i}(\mu (B_i \cap T^{-(j-i)} B_j) -\mu(B_i) \mu (B_j))
\]
\[
\le \sum_{i=n^*}^{n}  c F(t^*) \frac{\log i}{i} +\sum_{i=n^*}^{n} \sum_{j=i+c\log i}^{n}(\mu (B_i \cap T^{-(j-i)} B_j) -\mu(B_i) \mu (B_j))
\]
\[
\le\frac{\eta}{2} (\log n)^2+K
\]
where  we have used to the estimate $(*)$ to bound the second term by a constant $K$ . Thus for $n>n^*$, $E[(S_n-E_n)^2] <\eta
E_n^2 + 2 K +n^* E_n$. Since $\eta$ was arbitrary and $E_n$ diverges this implies that given $\epsilon>0$ there exists an $N$ such that for all $n>N$, $E[(S_n-E_n)^2] <\epsilon E_n^2$. Thus
Lemma~\ref{lemma:seq} implies the result.

\end{proof}


\subsection{Applications of Theorem~\ref{thm:exp}.}

We now give a brief list of systems, beside Axiom A~\cite{Hirata},
which have been shown to have the property that for $\mu$ a.e.\ $p$ if
$ B_i=B(p,r_i)$ then
\[
\lim_{r_i\to 0} (\mu_{B_i})^{-1} \{ y\in B_i : \tau_{B_i} \mu (B_i) < t\} =F(t) 
\]
for a distribution function $F(t)$ such that $\lim_{t\to 0}
F(t)=0$. In the examples below $F(t)=1-e^{-t}$, an exponential law. We
abbreviate this property by saying that the system has an exponential
law for first return times to balls (recall this holds only for $\mu$
a.e.\ point).

\noindent {\it Dispersing billiard systems:} These systems satisfy
assumption (B) (as the invariant measure is equivalent to
Lebesgue). The class of dispersing billiards considered by
Young~\cite{Young} and Chernov~\cite{Chernov} have exponential decay
of correlations. An exponential law for first return times to balls
has been established by Gupta et al.~\cite{GHN} (and a Poisson
distribution for further visits to balls by Chazottes and
Collet~\cite{Chazottes_Collet}).

\noindent {\it Lozi maps}~\cite{Collet-Levy,Misiurewicz}: it is shown
in Gupta et al.~\cite{GHN} that assumption (B) and an exponential
return time law is satisfied by a broad class of Lozi mappings (which
have exponential decay of correlations).

\noindent {\it Compact group extensions of Anosov systems:}
Dolgopyat~\cite{Dolgopyat} has shown that compact group extensions of
Anosov diffeomorphisms are typically rapid mixing i.e.\ have
superpolynomial decay of correlations. These systems satisfy also
Assumption (B) as they are volume preserving. Gupta~\cite{Gupta} has
shown the existence of an exponential law for the first return times
to nested balls if the system is rapidly mixing.

\noindent {\it Interval maps:} one-dimensional non-uniformly expanding
maps with an absolutely continuous invariant probabiility measure
$\mu$ and density $h = \frac{d\mu}{dm}\in L^{1+\delta} (m)$ for some
$\delta>0$ satisfy condition (A) and (B). The class of maps considered
by Collet~\cite{Collet} have been shown to have an exponential law for
first return times to balls.

\section{Short return times and sequences $(B_i)$ such that $\mu (B_i)\ge (i\log i)^{-1}$}

In the theory of return time statistics and extreme value theory a
crucial role is played by short return times. Recent research 
has established that for certain chaotic  dynamical systems short returns
are rare in the sense that we call property (C):

(C) for $\mu$ a.e.\ $p\in X$ there exists $\eta>0$ and $k>1$ such
that for all $i$ sufficiently large 
\[
\mu (B_i (p) \cap T^{-r} B_i (p)
)\le i^{-1-\eta}
\]
for all $r=1,\ldots, \log^k (i)$.



The property (C) is a form of non-recurrence which implies certain
limit laws in return time statistics and extreme value theory, to our
knowledge first proved in the setting of non-uniformly expanding maps
by Collet~\cite{Collet}.  It has been verified for various systems,
for example Sinai dispersing billiards and Lozi mappings~\cite{GHN}.
We give the proof of Property (C) for Sinai dispersing billiards as an
appendix.  Beside Axiom A systems, dynamical systems for which
Property (C) have been established include:

(1) {\it Billiard maps for dispersing billiards without cusps:} for
example the systems described
in~\cite{Young,Chernov,Chernov_Markarian}.  For a proof of this
property see the Appendix.

(2) {\it Lozi maps~\cite{Collet-Levy,Misiurewicz}}: it is shown in
Gupta et al.~\cite{GHN} that assumption (C) is satisfied by a broad
class of Lozi mappings.  Gupta~\cite{Gupta} shows the same for toral
extensions of certain non-uniformly and uniformly expanding intervals
maps.

(3) {\it Certain one-dimensional maps:} Collet first established
Property (C) in the setting of 1-d non-uniformly expanding maps with
acip with exponential decay of correlations~\cite{Collet}. These
results have been generalized to a broad class of 1d maps, including
Lorenz like maps~\cite{GHN,HNT}.

\begin{thm}\label{thm:short-returns}
  Suppose $(T,X,\mu)$ has exponential decay of correlations and
  satisfies conditions (B) and (C).  Then for $\mu$ a.e.\ $p\in X$ if
  $B_i=B(p,r_i)$ is a sequence of decreasing balls about $p$ with $\mu
  (B_i) \ge (i\log i)^{-1}$ then
  \[
  \lim_{n\rightarrow \infty} \frac{S_n (x)}{E_n}=1
  \]
  for $\mu$ a.e.\ $x\in X$, where as before $S_n=\sum_{j=0}^{n-1} 1_{B_i}$
  and $E_n=E(S_n)=\sum_{j=0}^{n-1}\mu (B_j)$
\end{thm}

\begin{cor} \label{cor:billiards}
If $(T,X,\mu)$ is a Sinai dispersing billiard map (see Appendix for
precise description) or a Lozi map (see~\cite{GHN}) then in the
notation of Theorem~\ref{thm:short-returns} for $\mu$ a.e.\ $p\in X$ if
$B_i$ is a sequence of decreasing balls about $p$ with $\mu (B_i) \ge
(i\log i)^{-1}$ then
\[
\lim_{n\rightarrow \infty} \frac{S_n (x)}{E_n}=1
\]
for $\mu$ a.e.\ $x\in X$.
\end{cor}

\begin{proof}[Proof of Corollary~\ref{cor:billiards}]
  The proof that Property (C) holds for Sinai dispersing billiard maps
  with $k=5$ is given in the Appendix. It is a very slight
  modification of the proof in Gupta et al.~\cite{GHN} where the
  sequence $\mu (B_i)=i^{-1}$ was considered.  A similar
  straightforward modification of the proof of the corresponding
  result for sequences $(B_i)$ such that $\mu (B_i)=i^{-1}$ given in
  Gupta et al.~\cite{GHN} in the setting of Lozi maps establishes the
  conclusions of the corollary for Lozi maps.
\end{proof}

\begin{proof}[Proof of Theorem~\ref{thm:short-returns}]
  We consider the case $\mu (B_i)=(i\log i )^{-1}$, which implies the
  corresponding result for $\mu (B_i)\ge (i\log i )^{-1}$.  We define
  $f_k$ and $g_k$ as before and split up $\sum_{j=i+1}^n E(f_i
  f_j)-E(f_i)E(f_j)$ into the terms $I_i$ and $II_i$ with $\Delta=[\log^2i]$.
  A rough estimate yields
  \begin{eqnarray*}
    I_i&=&\sum_{j=i+1}^{i+\log^2 i} E(f_i f_j)-E(f_i)E(f_j)\\
    &\le&\sum_{j=i+1}^{i+\log^2 i} E(f_i f_j) \\
    &\le & C (\log^2 i) (i\log i)^{-1} \\
    &\le& C i^{-1-(\eta/2)}
  \end{eqnarray*}
  for large $i$ and any positive $\eta$.
 For the second sum we obtain
  \begin{eqnarray*}
   II_i&=& \sum_{j=i+\log^2 i}^{n} E(f_i    f_j)-E(f_i)E(f_j)\\
    &\le&\sum_{\beta=1}^{\infty} C i^{2/\delta}
    (i+\log^2 i + \beta)^{2/\delta} \alpha^{\log^2 i + \beta}\\ &\le&
    C (i\log i)^{-1}
  \end{eqnarray*}
  for sufficiently large $i$. As above in light of Proposition~\ref{prop:sprindzuk}  this implies the  conclusion of the theorem.
\end{proof}

\section{Discussion.}

There are several questions that are prompted by this work. Here are some that we have considered but not resolved as yet.

\noindent (1)  If $(T,X,\mu)$ is a smooth dynamical system with acip and positive metric entropy is it true that for $\mu$ a.e.\ $p$
if $B_i=B(p,r_i)$ is a  nested sequence of balls about $p$ then 
\[
\lim_{r_i\to 0} (\mu_{B_i})^{-1} \{ y\in B_i : \tau_{B_i} \mu (B_i) < t\} =F_p(t) 
\]
for a distribution function $F_p(t)$ such that $\lim_{t\to 0} F_p(t)=0$?

\noindent (2) Chazottes and Collet~\cite{Chazottes_Collet} have shown that if $(T,X,\mu)$ is a dynamical system modeled
by a  Young tower, with exponential decay of correlations and a one-dimensional unstable foliation then for $\mu$ a.e.\ point 
$p$, if $B_i=B(p,r_i)$ is a  nested sequence of balls about $p$ then 
\[
\lim_{r_i\to 0} (\mu_{B_i})^{-1} \{ y\in B_i : \tau_{B_i} \mu (B_i) < t\} =1-e^{-t}
\]
In fact they have shown much more, including a Poisson law for multiple returns. If such systems satisfied Property (B) (respectively Property (C)) 
then the conclusion of Theorem~\ref{thm:exp}  (respectively Theorem~\ref{thm:short-returns}) would  hold. Does Property (B) or 
(C) hold for $\mu$ a.e.\ point in such systems?

\noindent (3) Is there an example of a smooth volume preserving dynamical system $(T,X,\mu)$ which has exponential decay of correlations yet
there is  a sequence of balls $B_i$, $\mu (B_i)\ge i^{-1}$ which is not Borel--Cantelli?

\subsection{Appendix: Property C for Planar Dispersing  Billiard Maps}\label{appendix}

We first describe the class of billiards for which we can prove Property C.  For  a good general reference to billiards
see~\cite{Chernov_Markarian}.

Let $\Gamma = \{\Gamma_i, i = 1, \ldots, k\}$ be a family of pairwise
disjoint, simply connected $C^3$ curves with strictly positive
curvature on the two-dimensional torus $\mathbb{T}^2$.  The billiard
flow $B_t$ is the dynamical system generated by the motion of a point
particle in $Q= \mathbb{T}^2/(\cup_{i=1}^k (\interior
\Gamma_i)$ with constant unit velocity inside $Q$ and with elastic
reflections at $\partial Q=\cup_{i=1}^k \Gamma_i$, where elastic means
``angle of incidence equals angle of reflection''.  If each $\Gamma_i$
is a circle then this system is called a periodic Lorentz gas. The
billiard flow is Hamiltonian and preserves a probability measure
(which is Liouville measure) $\tilde{\mu}$ given by
$d\tilde{\mu}=C_{Q} dq~dt$ where $C_{Q}$ is a normalizing constant and
$q\in Q$, $t\in \R$ are Euclidean coordinates.

We first consider the billiard map $T \colon \partial Q \to \partial
Q$.  Let $r$ be a one-dimensional co-ordinatization of $\Gamma$
corresponding to length and let $n(r)$ be the outward normal to
$\Gamma$ at the point $r$. For each $r\in \Gamma$ we consider the
tangent space at $r$ consisting of unit vectors $v$ such that
$(n(r),v)\ge 0$. We identify each such unit vector $v$ with an angle
$\theta \in [-\pi/2, \pi/2]$. The boundary $M$ is then parametrized by
$M:=\partial Q=\Gamma\times [-\pi/2, \pi/2]$ so that $M$ consists of
the points $(r,\theta)$. $T \colon M \to M$ is the Poincar\'e map that
gives the position and angle $T(r,\theta)=(r_1,\theta_1)$ after a
point $(r,\theta)$ flows under $B_t$ and collides again with $M$,
according to the rule angle of incidence equals angle of
reflection. Thus if $ (r,\theta)$ is the time of flight before
collision $T(r,\theta)=B_{h (r,\theta)} (r,\theta)$.  The billiard map
preserves a measure $d\mu=c_{M} \cos \theta dr d\theta$ equivalent to
$2$-dimensional Lebesgue measure $dm=dr\,d\theta$ with density $\rho
(x)$ where $x=(r,\theta)$.

We say that the billiard map and flows satisfies the finite horizon condition if 
the time of flight $h(r,\theta)$  is bounded above.
A good reference for background results for this section are the papers~\cite{BSC1,BSC2,Young, Chernov}.

It is known (see~\cite[Lemma 7.1]{Chernov} for finite horizon and
~\cite[Section 8]{Chernov} for infinite horizon) that dispersing
billiard maps expand in the unstable direction in the Euclidean metric
$|\cdot|=\sqrt{(dr)^2+(d\phi)^2}$ , in that $|DT_u^n v|\ge C
\tilde{\lambda}^n |v|$ for some constants $C$, $\tilde{\lambda}>1$
which is independent of $v$.  In fact $|L_n| \ge C\tilde{ \lambda}^n
|L_0|$ where $L_0$ is a segment of unstable manifold (once again in
the Euclidean metric) and $L_n$ is $T^n L_0$.

We choose $N_0$ so that $\lambda:= C\tilde{ \lambda}^{N_0}>1$ and then $T^{N_0} $ (or $DT^{N_0}$) expands
unstable manifolds (tangent vectors to unstable manifolds) uniformly in the Euclidean
metric.

It is common to use the $p$-metric in proving ergodic properties of
billiards. Young uses this semi-metric in~\cite{Young}. Recall that
for any curve $\gamma$, the $p$-norm of a tangent vector to $\gamma$
is given as $|v|_p =\cos\phi(r)|dr|$ where $\gamma$ is parametrized in
the $(r, \phi)$ plane as $(r, \phi(r)).$ The Euclidean metric in the
$(r, \phi)$ plane is given by $ds^2 = dr^2+d\phi^2$; this implies that
$|v|_p\le \cos\phi(r)ds\le ds = |v|$.  We will use $l_p (C)$ to denote
the length of a curve in the $p$-metric and $l(C)$ to denote length in
the Euclidean metric.  If $\gamma$ is a local unstable manifold or
local stable manifold then $C_1 l (\gamma )_p \le l(\gamma) \le C_2
\sqrt{ l_p (\gamma ) }$.

For planar dispersing billiards there exists an invariant measure
$\mu$ (which is equivalent to 2-dimensional Lebesgue measure) and
through $\mu$ a.e.\ point $x$ there exists a local stable manifold
$W_{loc}^s (x)$ and a local unstable manifold $W_{loc}^u (x)$.  The
SRB measure $\mu$ has absolutely continuous (with respect to Lebesgue
measure ) conditional measures $\mu_x$ on each $W_{loc}^u (x)$. The
expansion by $DT$ is unbounded however in the $p$-metric at
$\cos\theta=0$ and this may lead to quite different expansion rates at
different points on $W_{loc}^u (x)$. To overcome this effect and
obtain uniform estimates on the densities of conditional SRB measure
it is common to define homogeneous local unstable and local stable
manifolds. This is the approach adopted
in~\cite{BSC1,BSC2,Chernov,Young}. Fix a large $k_0$ and define for
$k>k_0$
\[
I_k=\{(r,\theta): \frac{\pi}{2} -k^{-2}
<\theta<\frac{\pi}{2}-(k+1)^{-2} \}
\]
\[
I_{-k}=\{(r,\theta): -\frac{\pi}{2} +(k+1)^{-2}
<\theta<-\frac{\pi}{2}+k^{-2} \}
\]
and 
\[
I_{k_0}=\{(r,\theta): -\frac{\pi}{2} +k_0^{-2} <\theta<\frac{\pi}{2}-
k_0^{-2} \}.
\]
In our setting we call a local unstable (stable) manifold $W^u_{loc}
(x)$, ($W^s_{loc} (x)$) homogeneous if for all $n\ge 0$ $T^n W^u_{loc}
(x)$ ($T^{-n} W^s_{loc} (x)$) does not intersect any of the line
segments in $\cup_{k>k_0} (I_k\cup I_{-k})\cup I_{k_0}$. Homogeneous
$W^u_{loc} (x)$ have almost constant conditional SRB densities
$\frac{d\mu_x}{dm_x}$ in the sense that there exists $C>0$ such that
$\frac{1}{C} \le \frac{d\mu_x (z_1)}{dm_x} /\frac{d\mu_x (z_2)}{dm_x}
\le C$ for all $z_1,~z_2 \in W^u_{loc} (x)$ (see ~\cite[Section
  2]{Chernov} and the remarks following Theorem 3.1).

From this point on all the local unstable (stable) manifolds that we
consider will be homogeneous. Bunimovich et al.~\cite[Appendix 2,
  Equation A2.1]{BSC2} give quantitative estimates on the length of
homogeneous $W^u_{loc} (x)$. They show there exists $C,~\tau >0$ such
that $\mu \{ x: l(W_{loc}^s (x))<\epsilon \mbox{ or } l(W_{loc}^u
(x))<\epsilon \} \le C \epsilon^{\tau}$ where $l(C)$ denotes
1-dimensional Lebesgue measure or length of a rectifiable curve $C$.
In our setting $\tau$ could be taken to be $\frac{2}{9}$, its exact
value will play no role but for simplicity in the forthcoming
estimates we assume $0<\tau<\frac{1}{2}$.

The natural measure $\mu$ has absolutely continuous conditional
measures $\mu_x$ on local unstable manifolds $W_{loc}^u (x)$ which
have almost uniform densities with respect to Lebesgue measure on
$W_{loc}^u (x)$ by~\cite[Equation 2.4]{Chernov}.

Let $A_{\sqrt{\epsilon}}=\{ x: |W_{loc}^u (x)|>\sqrt{\epsilon}\}$ then
$\mu (A^c_{\sqrt{\epsilon}})< C\epsilon^{\tau/2}$.  Let $x\in
A_{\sqrt{\epsilon}}$ and consider $W_{loc}^u (x)$. Since $|T^{-k}
W_{loc}^u (x)|<\lambda^{-1} |W_{loc}^u (x)|$ for $k>N_0$ the optimal
way for points $T^{-k}(y)$ in $T^{-k} W_{loc}^u (x)$ to be close to
their preimages $y\in W_{loc}^u (x)$ is for $T^{-k} W_{loc}^u (x)$ to
overlay $W_{loc}^u (x)$, in which case it has a fixed point and it is
easy to see $l\{y\in W_{loc}^u (x): d(y,T^{-k} y)<\epsilon\}\le
l\{y\in \R: d(y,\frac{y}{\lambda})<\epsilon\}\le
(1-\lambda^{-1})\epsilon$. Accordingly $l\{y\in W_{loc}^u (x):
d(y,T^{-k} y)<\epsilon\}\le C\sqrt{\epsilon} l\{y\in W_{loc}^u
(x)\}$. Recalling that the density of the conditional SRB-measure
$\mu_x$ is bounded above and below with respect to one-dimensional
Lebesgue measure we obtain $\mu_x (A^c_{\sqrt{\epsilon}})<
C\sqrt{\epsilon}$. Integrating over all unstable manifolds in
$A_{\sqrt{\epsilon}}$ (throwing away the set $\mu
(A^c_{\sqrt{\epsilon}})$) we have $\mu \{ x: d(T^{-k}x, x)<\epsilon)
<C \epsilon^{\tau/2}$. Since $\mu$ is $T$-invariant $\mu \{ x: d(T^k
x, x)<\epsilon\} <C \epsilon^{\tau/2}$ for $k>N_0$. Hence for any
iterate $T^k$, $k>N_0$
\[
\mathcal{E}_{k} (\epsilon):= \mu \{ x: d(T^k x, x)<\epsilon \} < C \epsilon^{\tau/2}.
\] 
Define 
\[
E_k:= \{x: d(T^j x,x)\le \frac{2}{\sqrt{k}}~\mbox{ for some }~1\le j \le (\log k )^5 \}.
\]
We have shown that for any $\delta>0$, for all sufficiently large $k$,
$\mu (E_k)\le k^{-\tau/4+\delta}$. For simplicity we take $\mu
(E_k)\le k^{-\sigma}$ where $\sigma <\tau/4 -\delta$.

Define the Hardy--Littlewood maximal function $M_l$ for $\phi(x)=
1_{E_l} (x)\rho(x)$ where $\rho(x)=\frac{d\mu}{dm} (x)$, so that
\[
 M_l(x):=\sup_{a>0}\frac{1}{m(B_a(x))}\int_{B_a(x)} 1_{E_l}(y)\rho(y) \, dm(y).
\]
A theorem of Hardy and Littlewood~\cite[Theorem 2.19]{mattila} implies that
\[
m( |M_l|>C)\le \frac{\|1_{E_l} \rho \|_1}{C}
\]
where $\|\cdot \|_1$ is the $L^1$ norm with respect to $m$. 
Let 
\[
F_k:=\{ x: \mu (B_{k^{-\gamma/2}} (x)\cap E_{k^{\gamma/2}})\ge
(k^{-\gamma\beta /2} )k^{\gamma/2}.
\]
Then $F_{k}\subset \{M_{k^{\gamma/2}}>k^{-\gamma\beta /2} \}$ and hence
\[
m(F_k) \le \mu (E_{k^{\gamma/2}})k^{\gamma\beta /2}\le C
k^{-\gamma\sigma}k^{\gamma\beta /2}.
\]
If we take $0<\beta <\sigma$ and $\gamma >\sigma/2$ then for some
$\delta>0$, $ k^{-\gamma\sigma}k^{\gamma\beta /2}<k^{-1-\delta}$ and
hence
\[
\sum_k m(F_k)<\infty.
\]
Thus for $m$ a.e.\ (hence $\mu$ a.e.) $x_0\in X$ there exists $N(x_0)$
such that $x_0 \not \in F_k$ for all $k>N(x_0)$.  Thus along the
subsequence $n_k=k^{-\gamma/2}$, $\mu (B_{n_k}(x_0)\cap
T^{-j}B_{n_k}(x_0) ) \le n_k^{-1-\delta}$ for $k>N(x_0)$.  This is
sufficient to obtain an estimate along the subsequence $(n\log
n)^{-1}$.  Since $\lim_{k\to \infty} (\frac{k+1}{k})^{\gamma/2} =1$ if
$k^{\gamma/2} \le n\log n \le (k+1)^{\gamma/2}$ then for sufficiently
large $n$ $\mu (B_{(n\log n)^{-1}}(x_0)\cap T^{-j}B_{(n\log
  n)^{-1}}(x_0) ) \le \mu (B_{n_k}(x_0)\cap T^{-j}B_{n_k}(x_0) )\le
n_k^{-1-\delta} \le 2(n\log n)^{-1-\delta}$.  As $\frac{d\mu}{dm} (p)
$ is finite for $\mu$ a.e.\ $p$ this implies the result for $\mu
(B_i)=i\log i$ by the Lebesgue density theorem.  We now control the
iterates $1\le j \le N_0$. If $x_0$ is not periodic then $\min_{1\le i
  <j \le N_0} d(T^i x_0, T^j x_0) \ge s(x_0)>0$ and hence for large
enough $n$, for all $1\le j \le N_0$, $\mu ( B_{n^{-1}}(x_0)\cap
T^{-j}B_{n^{-1}}(x_0))=0$.

\end{document}